\documentclass[a4paper,10pt]{amsart}

\usepackage[english]{babel}
\usepackage[utf8]{inputenc}

\usepackage{amssymb}
\usepackage{amsmath}
\usepackage{amsthm}
\usepackage{bbm}

\usepackage{enumerate}
\usepackage{tikz}
\usepackage{marginnote}

\newcommand{\bbC}{\mathbb{C}}

\newcommand{\bbN}{\mathbb{N}}

\newcommand{\bbR}{\mathbb{R}}

\newcommand{\bbT}{\mathbb{T}}

\newcommand{\calF}{\mathcal{F}}

\newcommand{\calU}{\mathcal{U}}

\DeclareMathOperator{\id}{id}

\newcommand{\argument}{\mathord{\,\cdot\,}}

\newcommand{\norm}[1]{\left\lVert #1 \right\rVert}
\newcommand{\modulus}[1]{\left\lvert #1 \right\rvert}

\newcommand{\spec}{\sigma}

\DeclareMathOperator{\trace}{tr}

\theoremstyle{definition}
\newtheorem{definition}{Definition}[section]

\newtheorem*{remark_no_number}{Remark}
\newtheorem*{remarks_no_number}{Remarks}
\newtheorem{example}[definition]{Example}

\theoremstyle{plain}

\newtheorem{theorem}[definition]{Theorem}

\newtheorem{conjecture}[definition]{Conjecture}

\numberwithin{equation}{section}

\begin{document}

\title[On the decoupled Markov group conjecture]{On the decoupled Markov group conjecture}
\author{Jochen Gl\"uck}
\address{Jochen Gl\"uck, Universität Passau, Fakultät für Informatik und Mathematik, 94032 Passau, Germany}
\email{jochen.glueck@uni-passau.de}
\date{\today}
\begin{abstract}
	The Markov group conjecture, a long-standing open problem in the theory of Markov processes with countable state space, asserts that a strongly continuous Markov semigroup $T = (T_t)_{t \in [0,\infty)}$ on $\ell^1$ has bounded generator if the operator $T_1$ is bijective. Attempts to disprove the conjecture have often aimed at glueing together finite dimensional matrix semigroups of growing dimension -- i.e., it was tried to show that the Markov group conjecture is false even for Markov processes that decouple into (infinitely many) finite dimensional systems.
	
	In this article we show that such attempts must necessarily fail, i.e., we prove the Markov group conjecture for processes that decouple in the way described above. In fact, we even show a more general result that gives a universal norm estimate for bounded generators $Q$ of positive semigroups on any Banach lattice.
	
	Our proof is based on a filter product technique, infinite dimensional Perron--Frobenius theory and Gelfand's $T = \id$ theorem.
\end{abstract}

\maketitle

\section{Introduction} \label{section:introduction}

\subsection*{The Markov group conjecture and its decoupled version}

In 1967, the following problem was posed and partially analysed by Kendall \cite{Kendall1967} and Speakman \cite{Speakman1967}.

\begin{conjecture}[Markov group conjecture] \label{conj:markov-groups}
	Let $T = (T_t)_{t \in [0,\infty)}$ be a Markovian $C_0$-semigroup on $\ell^1$ and assume that $T_1: \ell^1 \to \ell^1$ is bijective (i.e., $T$ extends to a $C_0$-group). Then $T$ has bounded generator.
\end{conjecture}

Here, \emph{Markovian} (or \emph{Markov}) means that, for each $t \ge 0$, the operator $T_t$ is positive (in the sense that $T_t x \ge 0$ for all $x \ge 0$) and norm-preserving on the positive cone.

For a few classes of semigroups the conjecture is easy to prove (see \cite[Section~3]{Kendall1967}), and in the first years after the formulation of the conjecture, partial results were obtained by various authors \cite{Williams1969, Cuthbert1972, Cuthbert1975, Mountford1977}. Afterwards though, progress on the problem has been slow. An overview of the problem was given by Kingman on several occassions; see \cite[Section~2]{Kingman1983}, \cite{Kingman2006}, \cite[Section~9]{Kingman2006a}. In attempts to find a counterexample, a common approach is to consider finite dimensional matrices $Q_n$ that generate Markov semigroups on $\bbR^{d_n}$ such that $\norm{e^{-Q_n}} \le M$ for all indices $n$ and a fixed constant $M$. If one succeeded in choosing $Q_n$ such that $\norm{Q_n} \to \infty$, the block diagonal operator on $\ell^1$ with block entries $Q_n$ would generate a Markov semigroup on $\ell^1$ that disproves the conjecture. 

Such direct sum semigroups were already considered in the original papers by Kendall and Speakman \cite{Kendall1967, Speakman1967}, and the strategy to use them for constructing a counterexample was further discussed by Kingman in \cite[Section~2]{Kingman1983} and \cite[Sections~2 and~3]{Kingman2006}. Phrased in other words, the goal of this strategy is to find a counterexample to the following slightly weaker conjecture. Motivated by the diagonal construction described above, one could call it the \emph{decoupled Markov group conjecture}.

For each $d \in \bbN$, endow $\bbR^{d \times d}$ with the operator norm induced by the $1$-norm on $\bbR^d$.

\begin{conjecture} \label{conj:markov-groups-finite-uniform}
	Let $M \ge 1$ be a real number. Then there exists a real number $C = C(M) \ge 0$ with the following property:
	
	For every $d \in \bbN$ and for every matrix $Q \in \bbR^{d \times d}$ that satisfies $\norm{e^{-Q}} \le M$ and whose associated matrix semigroup $(e^{tQ})_{t \in [0,\infty)}$ is column stochastic, we have $\norm{Q} \le C$.
\end{conjecture}

This conjecture was explicitely formulated by Kingman in \cite[p.\ 186]{Kingman1983}.

\begin{remarks_no_number}
	\begin{enumerate}[(a)]
		\item The main point of Conjecture~\ref{conj:markov-groups-finite-uniform} is that $C(M)$ does not depend on the dimension $d$. It is very easy to prove a dimension dependent estimate: namely, we have
		\begin{align*}
			\norm{Q} \le 2 \modulus{\trace{Q}} \le 2d \log(M)
		\end{align*}
		for each matrix $Q \in \bbR^{d \times d}$ that satisfies the assumptions in Conjecture~\ref{conj:markov-groups-finite-uniform} (use that $e^{\modulus{\trace Q}} = e^{-\trace Q} = \det(e^{-Q}) \le M^d$).

		\item Note that, if $(e^{tQ})_{t \in [0,\infty)}$ is column stochastic and $M \ge 1$, then the estimate $\norm{e^{-Q}} \le M$ is equivalent to $\sup_{t \in [-1,\infty)} \norm{e^{tQ}} \le M$. When we discuss bounded positive semigroups (rather than only column stochastic ones) below, we will use this latter estimate rather than $\norm{e^{-Q}} \le M$.
	\end{enumerate}
\end{remarks_no_number}

\subsection*{Main result}

The main objective of this paper is to prove Conjecture~\ref{conj:markov-groups-finite-uniform}. In fact, though, we show a much stronger result which has nothing to do with the finite dimensional spaces $\bbR^d$ nor with choice of the $1$-norm on them. We prove:

\begin{theorem} \label{thm:main-result}
	Let $M \ge 1$ be a real number. Then there exists a universal constant $C = C(M) \in [0,\infty)$ (depending solely on $M$) with the following property:
	
	For every complex Banach lattice $E$ and every bounded linear operator $Q$ on $E$ that satisfies
	\begin{align}
		\label{eq:norm-estimate}
		\sup_{t \in [-1,\infty)} \norm{e^{tQ}} \le M
	\end{align}
	and whose associated semigroup $(e^{tQ})_{t \in [0,\infty)}$ is positive, we have $\norm{Q} \le C$.
\end{theorem}

The essence of the theorem is: if one knows a priori that the generator $Q$ of a bounded positive $C_0$-semigroup is bounded, then one can estimate the norm $\norm{Q}$ by a constant that merely depends on the number $\sup_{t \in [-1,\infty)} \norm{e^{tQ}}$.

\subsection*{Relation to the Markov group conjecture}

On finite dimensional spaces all operators are bounded, so Theorem~\ref{thm:main-result} implies that Conjecture~\ref{conj:markov-groups-finite-uniform} is true, and we conclude that one cannot disprove the Markov group conjecture~\ref{conj:markov-groups} by using a block diagonal construction that consists of finite dimensional blocks (or, more generally, of blocks that have bounded generator).

It is not immediately clear (at least not to the author) whether the Markov group conjecture~\ref{conj:markov-groups} follows from Theorem~\ref{thm:main-result}. It was mentioned by Kingman in \cite[pages 186-187]{Kingman1983} that it might be possible to derive Conjecture~\ref{conj:markov-groups} from the a priori weaker statement in Conjecture~\ref{conj:markov-groups-finite-uniform} by means of approximation, but in a later paper the same author noted that it is actually not clear whether~\ref{conj:markov-groups} and~\ref{conj:markov-groups-finite-uniform} are equivalent \cite[beginning of Section~4]{Kingman2006}.

If there is indeed a way to derive the Markov group conjecture~\ref{conj:markov-groups} from its decoupled version~\ref{conj:markov-groups-finite-uniform} or more generally from Theorem~\ref{thm:main-result} by means of approximation, this endeavour is necessarily subject to considerable theoretical restrictions; see the remark at the end of Section~\ref{section:proof-of-main-result} for details.

In \cite[page 6]{Kingman2006} Kingman asked whether the assertion of Theorem~\ref{thm:main-result} holds if one only considers the single infinite-dimensional Banach lattice $\ell^1$ (in the notation of \cite[Section~3]{Kingman2006}, he asked whether $K(m) < \infty$ for number each $m > 1$); Theorem~\ref{thm:main-result} shows that the answer is positive.

\subsection*{Organization of the paper}

We prove Theorem~\ref{thm:main-result} in Section~\ref{section:proof-of-main-result}. In Section~\ref{section:non-positive-semigroups} we briefly explain that a similar result also holds for certain classes of non-positive semigroups on $L^p$ if $p \not= 2$. In the appendix we briefly recall a few facts about filter products of Banach spaces; these are needed in the proof of our main result.

\subsection*{Prerequisites}

We assume that reader to be familiar with the basic theories of $C_0$-semigroups (see for instance \cite{Engel2000}) and Banach lattices (see for instance \cite{Schaefer1974} and \cite{Meyer-Nieberg1991}). We call a linear operator $T$ on a Banach lattice \emph{positive} if $Tf \ge 0$ whenever $f \ge 0$ (i.e., no strict positivity is required in any sense).

\section{Proof of the main result} \label{section:proof-of-main-result}

The subsequent proof uses that concept of a \emph{filter product} of a sequence of Banach lattices. Readers not familiar with this technology can find a (very) brief introduction, as well as several references, in Appendix~\ref{appendix:reminder-of-filter-products}.

\begin{proof}[Proof of Theorem~\ref{thm:main-result}]
	Fix $M$ and assume that such a constant $C = C(M)$ does not exist. Then we can find a sequence of complex Banach lattices $E_n$ and a sequence of bounded linear operator $Q_n$ on $E_n$ such that: each $Q_n$ generates a positive semigroup on $E_n$, each $Q_n$ satisfies the norm estimate~\eqref{eq:norm-estimate} and each $Q_n$ has norm $\norm{Q_n} \ge n$. We set $R_n := \frac{Q_n}{\norm{Q_n}}$ for each $n$.
	
	Let $\calF$ denote the Fr\'{e}chet filter on $\bbN$ (or any other Filter which is finer than the Fr\'{e}chet filter) and let $E := (E_n)^\calF$ denote the $\calF$-product of the spaces $E_n$ (see Appendix~\ref{appendix:reminder-of-filter-products}). Then $E$ is a complex Banach lattice. We define $R := (R_n)^\calF$, i.e., $R$ is the bounded linear operator on $E$ given by $R(x_n)^\calF = (R_nx_n)^\calF$ for each norm bounded sequence $(x_n)$ of vectors $x_n \in E_n$. Since each $R_n$ has norm $1$, we also have $\norm{R} = 1$.
	
	We now derive a contradiction by showing that we must actually have $R = 0$. To this end, observe that
	\begin{align*}
		e^{tR} = (e^{tR_n})^\calF
	\end{align*}
	for all $t \in \bbR$. For each $t \in [0,\infty)$ and each $n \in \bbN$ we note that the operator $e^{tR_n} = e^{\frac{t}{\norm{Q_n}}Q_n}$ is positive and has norm at most $M$; hence, $e^{tR}$ is positive and satisfies $\norm{e^{tR}} \le M$ for each $t \in [0,\infty)$. Therefore, every spectral value of $R$ has real part $\le 0$, and it follows from infinite-dimensional Perron--Frobenius theory that $\spec(R) \cap i \bbR \subseteq \{0\}$ (see \cite[Corollary~C-III-2.13]{Nagel1986}). 
	
	Now comes the essential point: we claim that the group $(e^{tR})_{t \in \bbR}$ is also bounded for negative times. To see this, let $t > 0$. For every index $n \ge t$ we then have $\norm{Q_n} \ge n \ge t$, so
	\begin{align*}
		\norm{e^{-tR_n}} = \norm{e^{-\frac{t}{\norm{Q_n}} Q_n}} \le M
	\end{align*}
	since $-\frac{t}{\norm{Q_n}} \in [-1,0]$ and since $Q_n$ satisfies~\eqref{eq:norm-estimate}. On the $\calF$-product $E$, only the norms for large indices $n$ matter, so $\norm{e^{-tR}} \le M$. As $t > 0$ was arbitrary, the group $(e^{tR})_{t \in \bbR}$ is indeed bounded.
	
	Thus, the spectrum $\spec(R)$ is a subset of the imaginary axis and therefore, $\spec(R) = \{0\}$. Now we use the boundedness of the group $(e^{tR})_{t \in \bbR}$ a second time: as $\spec(R) = \{0\}$, it follows from the $C_0$-group version if Gelfand's $T = \id$ theorem \cite[Corollary~4.4.11]{Arendt2011} that $e^{tR} = \id_E$ for each time $t$. So $R = 0$, a contradiction.
\end{proof}

\begin{remarks_no_number}
	\begin{enumerate}[(a)]
		\item The arguments in the proof above that use Perron-Frobenius theory and Gelfand's $T = \id$ theorem actually show that, if a bounded linear operator $Q$ generates a bounded group $(e^{tQ})_{t \in \bbR}$ which is positive for $t \ge 0$, then $Q = 0$. Hence, we can choose $C(1) = 0$ in the theorem.
		
		\item The proof of Theorem~\ref{thm:main-result} demonstrates how infinite dimensional methods can be of use to solve finite dimensional problems: if the spaces $E_n$ in the proof are $\bbC^{d_n}$ and endowed with the $1$-norm (meaning that we prove Conjecture~\ref{conj:markov-groups-finite-uniform} rather than the more general Theorem~\ref{thm:main-result}), their $\calF$-product $E$ will still be an infinite dimensional Banach lattice. Even if we replace $\calF$ with an ultrafilter, $E$ will be infinite dimensional unless the dimensions $d_n$ are bounded.
		
		\item In a sense, Theorem~\ref{thm:main-result} can be considered as a Tauberian theorem: on a class of operators, we consider the transformation
		\begin{align*}
			Q \; \mapsto \; \Big([-1,\infty) \ni t \mapsto e^{tQ}\Big)
		\end{align*}
		which maps each operator to a certain operator-valued function. Theorem~\ref{thm:main-result} then says that $Q$ can be bounded uniformly by a norm bound of its transform. This interpretation of Theorem~\ref{thm:main-result} was kindly brought to my attention by Wolfgang Arendt.
		
		\item A similar approach as in the proof above, using Perron--Frobenius theory and Gelfand's $T = \id$ theorem to show that a given semigroup generator equals $0$, was used in \cite[Section~2]{Glueck2018} to give a new proof of a classical result of Sherman about lattice ordered $C^*$-algebras. 
		
		The same comments as at the end of \cite[Section~2]{Glueck2018} also apply to the proof above; in particular:
		
		\item The Perron--Frobenius type theorem from \cite[Corollary~C-III-2.13]{Nagel1986} that we used in the proof relies on quite heavy machinery. However, we only use the result for semigroups with bounded generators, for which it is much simpler to prove -- see for instance \cite[Proposition~2.2]{Glueck2018}.
		
		\item Our proof also uses Gelfand's $T = \id$ theorem for $C_0$-semigroups which is not quite trivial. But again, we apply this theorem only for semigroups with bounded generator -- and for these, it can be derived from the single operator version of Gelfand's $T = \id$ theorem, which is a bit simpler (see for instance \cite[Theorem~1.1]{Allan1989}).
	\end{enumerate}
\end{remarks_no_number}

Let us comment once again on the connection between Theorem~\ref{thm:main-result} and the Markov group conjecture~\ref{conj:markov-groups}.

\begin{remark_no_number}
	The following approach to the Markov group conjecture is tempting: given the $C_0$-semigroup $T$ in the conjecture, we could try to approximate it by a sequence of semigroups $T_n$ which are, say, also (sub-)Markovian (or at least positive and uniformly bounded) and which have bounded generators $Q_n$. If we manage to choose this approximation such that $\norm{e^{-tQ_n}} \le M$ for all indices $n$, then Theorem~\ref{thm:main-result} implies that $\norm{Q_n} \le C(M)$ for all $n$, and from this we can derive that the generator $Q$ of $T$ is bounded, too (provided that the approximation is sufficiently reasonable in the sense that the $Q_n$ converge to $Q$, say strongly on the domain of $Q$). This approach is also discussed by Kingman at the beginning of \cite[Section~4]{Kingman2006}.
	
	Let us now explain how Theorem~\ref{thm:main-result} provides a new perspective on this idea. The discrete structure of $\ell^1$ is, of course, essential for the Markov group conjecture, since the conjecture is false on other $L^1$-spaces (consider for instance the rotation group on $L^1(\bbT)$, where $\bbT$ denotes the complex unit circle). So where does discreteness enter the game?
	
	For the application of Theorem~\ref{thm:main-result}, the discrete structure of $\ell^1$ does not matter since we proved the theorem for all Banach lattices. Hence, it is necessarily the approximation procedure where the discreteness of $\ell^1$ has to be used. So if the approximation approach is supposed to work, either the construction of the approximation itself or the proof of the property $\norm{e^{-Q_n}} \le M$ has to make use of the discreteness of $\ell^1$ in a fundamental way.
	
	Note that classical approximations, such as the ones of Hille and Yosida (see \cite[Section~II-3.3]{Engel2000}), work on any Banach space. So we conclude that either such approximation procedures cannot be used in the approach discussed above, or the discreteness of $\ell^1$ has to be used to show that such a procedure allows an estimate of the type $\norm{e^{-Q_n}} \le M$ (which is not true for the Hille and the Yosida approximation on general $L^1$-spaces, as can again be seen be considering the rotation group on $L^1(\bbT)$).
\end{remark_no_number}

\section{On non-positive semigroups} \label{section:non-positive-semigroups}

The only step in the proof of Theorem~\ref{thm:main-result} where we needed positivity of the semigroups was the application of a Perron--Frobenius type result to derive that the spectrum of $R$ intersects $i\bbR$ at most in $0$. There are, however, similar results for certain classes of non-positive semigroups:

Let $p \in [1,\infty)$, but $p \not= 2$, and consider the complex-valued space $L^p$ over an arbitrary measure space. If $A$ is the generator of a contractive, real and eventually norm continuous $C_0$-semigroup on $L^p$, then $\spec(A) \cap i\bbR \subseteq \{0\}$; this was proved in \cite[Corollary~4.6 and Remark~4.8(i)]{Glueck2016}. (By \emph{real}, we mean that the semigroup operators map real-valued functions to real-valued functions; by \emph{contractive}, we mean that every semigroup operator has norm at most $1$.) So we can deduce the following theorem.

\begin{theorem} \label{thm:contractive}
	Fix $p \in [1,\infty) \setminus \{2\}$ and a real number $M \ge 1$. Then there exists a universal constant $C = C(p,M) \in [0,\infty)$ (depending solely on $p$ and $M$) with the following property:
	
	For every $L^p$-space (over an arbitary measure space) and every bounded linear operator $Q$ on $L^p$ that satisfies $\norm{e^{-Q}} \le M$ and whose associated semigroup $(e^{tQ})_{t \in [0,\infty)}$ is real and contractive, we have $\norm{Q} \le C$.
\end{theorem}

We point out that the semigroup generated by $Q$ is real if and only if $Q$ itself is real. Our proof of Theorem~\ref{thm:contractive} uses spectral theory, and thus complex $L^p$-spaces. However, the theorem holds for real-valued $L^p$-spaces as well, even with the same constant $C(p,M)$. This follows from the fact that the complex extension of a bounded linear operator $T$ on a real-valued $L^p$-space has the same norm as $T$ itself \cite[Proposition~2.1.1]{Fendler1998}.

\begin{proof}[Proof of Theorem~\ref{thm:contractive}]
	The argument is very similar to the proof of Theorem~\ref{thm:main-result}, with two simple changes:
	\begin{enumerate}[(1)]
		\item The spaces $E_n$ are now $L^p$-spaces, and we need their filter product $E$ to be an $L^p$-space, too. Thus, we have to replace the Fr\'{e}chet filter $\calF$ with a free ultrafilter $\calU$ on $\bbN$ (see Subsection~\ref{subsection:ultrafilters} in the appendix).
		
		\item Instead of Perron-Frobenius theory, we now derive the fact $\spec(R) \cap i \bbR \subseteq \{0\}$ from the results quoted before Theorem~\ref{thm:contractive}. This works since $E$ is an $L^p$-space for $p\not= 2$ and since the ultraproduct of real operators is again real.
	\end{enumerate}
	The rest of the proof is the same.
\end{proof}

\begin{remark_no_number}
	Theorems~\ref{thm:main-result} and~\ref{thm:contractive} actually yield two independent reasons for Conjecture~\ref{conj:markov-groups-finite-uniform} to be true: 
	
	Theorem~\ref{thm:main-result} implies the conjecture since every column stochastic semigroup is positive and bounded. Indepently of that, Theorem~\ref{thm:contractive} implies the conjecture since overy column stochastic semigroup is real and contractive with respect to the $1$-norm on $\bbR^d$.
\end{remark_no_number}

We conclude the paper with the following simple example which demonstrates why the positivity assumption cannot be dropped in Theorem~\ref{thm:main-result} (without any replacement) and why the assumption $p \not= 2$ cannot be dropped in Theorem~\ref{thm:contractive} -- not even for finite dimensional spaces with fixed dimension.

\begin{example} \label{ex:rotation-in-2d}
	Endow $\bbC^2$ with the Euclidean norm. For each $n \in \bbN$, consider the $2 \times 2$-matrix
	\begin{align*}
		Q_n :=
		\begin{pmatrix}
			0 & -n \\
			n & \phantom{-}0
		\end{pmatrix}.
	\end{align*}
	It has spectrum $\{-in,in\}$, and its operator norm (induced by the Euclidean norm on $\bbC^2$) is $\norm{Q_n} = n$. The matrix $Q_n$ generates the two-dimensional rotation group that is given by
	\begin{align*}
		e^{tQ_n} =
		\begin{pmatrix}
			\cos(nt) & -\sin(nt) \\
			\sin(nt) & \phantom{-} \cos(nt)
		\end{pmatrix}
	\end{align*}
	for each time $t \in \bbR$. Hence, $\norm{e^{tQ_n}} = 1$ for all $t \in \bbR$ and all $n \in \bbN$, so we cannot bound $\norm{Q_n} = n$ by a constant multiple of $\sup_{t \in [-1,\infty)} \norm{e^{tQ_n}} = 1$.
	
	The reason why Theorem~\ref{thm:main-result} cannot be applied is that the semigroup $(e^{tQ_n})_{t \in [0,\infty)}$ is not positive, and Theorem~\ref{thm:contractive} cannot be applied since the semigroup is not contractive with respect to the $p$-norm for any $p \not= 2$.
\end{example}

\subsection*{Acknowledgements} 

It is my pleasure to thank Markus Haase for bringing the Markov group conjecture to my attention.

\appendix 

\section{A brief reminder a filter products} \label{appendix:reminder-of-filter-products}

Filter products, and in particular ultraproducts, are a powerfool and widely used tool in Banach space and operator theory; details about ultraproducts can, for instance, be found in the survey article \cite{Heinrich1980} and in \cite[Chapter~8]{Diestel1995}. For examples of the use of such techniques in operator theory and, in particular, in spectral theory, we refer to \cite[Sections~V.1 and~V.4]{Schaefer1974} and \cite[Section~4.1]{Meyer-Nieberg1991}.

For the proof of Theorem~\ref{thm:main-result} we do not really need ultraproducts; products with respect to the Fr\'{e}chet filter suffice (although the proof works just as well with ultraproducts), and we briefly outline the construction of such Fr\'{e}chet filter products in Subsections~\ref{subsection:filter-products} and~\ref{subsection:operators} below. Ultraproducts are essential for the proof of Theorem~\ref{thm:contractive} and are briefly explained in Subsection~\ref{subsection:ultrafilters}.

\subsection{Filter products of Banach spaces and Banach lattices} \label{subsection:filter-products}

Let $\calF \subseteq 2^\bbN$ denote the \emph{Fr\'{e}chet filter} on $\bbN$, i.e., the filter that consists of all subsets of $\bbN$ with finite complement. The construction of an $\calF$-product of Banach spaces works as follows.

Let $(E_n)$ be a sequence of Banach spaces (over the same scalar field) and let $E^\infty$ denote the space of all sequences $x = (x_n)$ such that $\norm{x}_\infty := \sup_{n \in \bbN} \norm{x_n} < \infty$. Then $(E^\infty, \norm{\argument}_\infty)$ is also a Banach space. Now we wish to ``factor out the behaviour at finite indices''; more precisely, we consider $E_0 := \{x \in E^\infty: \, \norm{x_n} \to 0\}$, which is a closed subspace of $E^\infty$. The $\calF$-product of the spaces $(E_n)$ is defined to be the quotient space
\begin{align*}
	E^\calF := E^\infty / E_0.
\end{align*}
The notation $E^\calF$ and the notion ``filter product'' might be surprising at first glance, since we did not use $\calF$ explicitly in the construction of $E^\calF$; we explain the relevance of the filter $\calF$ in Subsection~\ref{subsection:ultrafilters}.

For each sequence $(x_n) \in E^\infty$ we use the notation $(x_n)^\calF$ to denote the equivalence class of $(x_n)$ in $E^\calF$; it is not difficult to see that the (quotient) norm of $(x_n)^\calF$ equals $\limsup_{n \to \infty} \norm{x_n}$.

If each space $E_n$ is a (real or complex) Banach lattice, then so is $E^\infty$ (with the pointwise ordering), and $E_0$ is then an ideal in $E^\infty$. Hence, the quotient space $E^\calF$ is a Banach lattice, too.

\subsection{Operators} \label{subsection:operators}

Assume that we are given a sequence of bounded linear operators $T_n$ on the Banach spaces $E_n$, such that $\sup_{n \in \bbN} \norm{T_n} < \infty$. Then we can define an operator $T^\infty$ on $E^\infty$ by 
\begin{align*}
	T^\infty (x_n) = (T_n x_n)
\end{align*}
for each sequence $(x_n) \in E^\infty$. This operator $T^\infty$ clearly leaves $E_0$ invariant, so it induces an operator $(T_n)^\calF$ on the filter product $E^\calF$ that is given by
\begin{align*}
	(T_n)^\calF (x_n)^\calF = (T_n x_n)^\calF
\end{align*}
for each $(x_n)^\calF \in E^\calF$ (with $(x_n) \in E^\infty$). The norm of $T^\calF$ is easy to compute; it is given by $\norm{(T_n)^\calF} = \limsup_{n \to \infty} \norm{T_n}$. 

For two bounded operator sequences $(T_n)$ and $(S_n)$ and scalars $\alpha,\beta$ we have $(\alpha T_n + \beta S_n)^\calF = \alpha (T_n)^\calF + \beta (S_n)^\calF$ and $(S_n T_n)^\calF = (S_n)^\calF (T_n)^\calF$.

If all the spaces $E_n$ are Banach lattices and each operator $T_n$ is positive, then $(T_n)^\calF$ is positive, too.

\subsection{Ultrafilters and ultraproducts} \label{subsection:ultrafilters}

The construction outlined in Subsections~\ref{subsection:filter-products} and~\ref{subsection:operators} is completely sufficient for the proof of Theorem~\ref{thm:main-result}, and it does note use the filter $\calF$ in any explicit way. So why do we insist on this terminology and notation?

The problem about the space $E^\calF$ is that it does not respect any regularity or geometric property of the spaces $E_n$. Even if all the spaces $E_n$ are one-dimensional, the space $E^\calF$ will be an infinite dimensional, non-separable and non-reflexive Banach space. This is not good enough for the proof of our second result, Theorem~\ref{thm:contractive}. Here is where the filters enter the game:

As $\calF$ denotes the Fr\'{e}chet filter on $\bbN$, the space $E_0$ can also be written as $E_0 := \{x \in E^\infty: \; \lim_{n \to \calF} \norm{x_n} = 0\}$ (hence the notation $E^\calF$ and the name \emph{filter product} for $E^\infty / E_0$). But this expression makes sense not only for the Fr\'{e}chet filter  $\calF$, but also for every filter that is finer than $\calF$; in particular, it makes sense for every free ultrafilter on $\bbN$. So if we replace $\calF$ with a free ultrafilter $\calU$ and repeat the construction outlined above, we end up with a space $E^\calU$, which is referred to as an \emph{ultraproduct} of the spaces $E_n$.

The use of ultrafilters has a major advantage compared to the Fr\'{e}chet filter: every bounded sequence in $\bbR$ converges along every ultrafilter, and from this one can easily derive that we now have
\begin{align*}
	\norm{(x_n)^\calU} = \lim_{n \to \calU} \norm{x_n}
\end{align*}
for each $(x_n)^\calU \in E^\calU$ -- i.e., the $\limsup$ from Subsection~\ref{subsection:filter-products} has now been replaced with a limit. This ensures that many geometric properties of Banach spaces are respected by ultraproducts. For instance, it easily follows that, for fixed $p \in [1,\infty)$, the norm on an ultraproduct $E^\calU$ of $L^p$-spaces is $p$-additive, and thus $E^\calU$ is itself an $L^p$-space by the representation theorem for Banach lattices with $p$-additive norm \cite[Theorem~2.7.1]{Meyer-Nieberg1991}. This is what we need in the proof of Theorem~\ref{thm:contractive}.

\bibliographystyle{plain}
\bibliography{markov_groups}

\end{document}